\documentclass[10pt]{amsart}

\usepackage[numbers,sort]{natbib}
\usepackage{amsmath,amssymb,amsthm,mathtools,graphicx,microtype,xcolor,charter}
\usepackage{bm}
\usepackage{array,hyperref,enumitem}
\usepackage{orcidlink,cleveref,float}
\usepackage{booktabs}
\usepackage[margin=.91in]{geometry}
\hypersetup{
  hidelinks,
  pdftitle={Threshold dynamics for subdiffusive grain growth},
  pdfauthor={Marvin Fritz}
}

\newcolumntype{M}[1]{>{\centering\arraybackslash}m{#1}}

\theoremstyle{plain}
\newtheorem{theorem}{Theorem}[section]
\newtheorem{lemma}[theorem]{Lemma}

\newtheorem{proposition}[theorem]{Proposition}

\theoremstyle{definition}
\newtheorem{definition}[theorem]{Definition}

\theoremstyle{remark}
\newtheorem{remark}[theorem]{Remark}

\crefname{theorem}{theorem}{theorems}
\Crefname{theorem}{Theorem}{Theorems}
\crefname{lemma}{lemma}{lemmas}
\Crefname{lemma}{Lemma}{Lemmas}
\crefname{corollary}{corollary}{corollaries}
\Crefname{corollary}{Corollary}{Corollaries}
\crefname{proposition}{proposition}{propositions}
\Crefname{proposition}{Proposition}{Propositions}
\crefname{definition}{definition}{definitions}
\Crefname{definition}{Definition}{Definitions}
\crefname{example}{example}{examples}
\Crefname{example}{Example}{Examples}
\crefname{remark}{remark}{remarks}
\Crefname{remark}{Remark}{Remarks}
\crefname{notation}{notation}{notations}
\Crefname{notation}{Notation}{Notations}
\crefname{assumption}{assumption}{assumptions}
\Crefname{assumption}{Assumption}{Assumptions}
\crefname{problem}{problem}{problems}
\Crefname{problem}{Problem}{Problems}

\newcommand{\Om}{\Omega}

\newcommand{\R}{\mathbb{R}}
\newcommand{\T}{\mathbb{T}}
\newcommand{\Caputo}{\partial_t^\alpha}
\newcommand{\Gammafun}{\Gamma}
\newcommand{\1}{\mathbf{1}}
\newcommand{\dd}{\,\textup{d}}

\begin{document}

\title[Threshold dynamics for subdiffusive grain growth]{\Large Threshold dynamics for subdiffusive grain growth}

\author[Marvin Fritz]{Marvin Fritz\orcidlink{0000-0002-8360-7371}}
\address{Faculty of Mathematics, University of Vienna, Vienna, Austria}

\begin{abstract}
We propose a new fractional threshold-dynamics algorithm for subdiffusive interface motion and apply it to multiphase grain-growth simulations. The method combines the L1 discretization of the Caputo time derivative with a Helmholtz-resolvent thresholding step. At each time level, the L1 convolution produces a convex history average of the previously thresholded states; this history-dependent field is then propagated by a Helmholtz diffusion solve and projected back to pure phases by thresholding. For $\alpha=1$, the method reduces to a Helmholtz-resolvent analogue of the Merriman--Bence--Osher scheme, whereas for $0<\alpha<1$ it introduces {a subdiffusive temporal memory}.
{We establish a variational characterization, maximum and comparison principles, and stability for the continuous Helmholtz diffusion step. We also derive a memory-tail estimate and a finite-dimensional pinning criterion under an explicit resolvent-contraction hypothesis. The equal-tension multiphase scheme is implemented on periodic grids. The numerical experiments investigate the joint influence of the L1 history and the smoothing length. The fixed-step runs show faster coarsening for smaller fractional orders in the tested regime, while a comparison at fixed smoothing length examines the influence of the L1 history on successive updates.}
\end{abstract}



\maketitle
\vspace{-.5cm}

\section{Introduction}\label{sec:intro}

The classical
Merriman--Bence--Osher (MBO) algorithm evolves a set by alternating a heat step
and a thresholding step, and it approximates motion by mean curvature in the
sharp-interface limit \cite{MerrimanBenceOsher1992,Evans1993,BarlesGeorgelin1995}.
This threshold scheme can be interpreted formally as a singular
operator splitting of the Allen--Cahn equation: the heat part is evolved, while the fast
reaction toward the pure phases is replaced by instantaneous thresholding \cite{Modica1987,Ilmanen1993,BronsardKohn1991,LauxSimon2018}.

Our main
object is a history-dependent thresholding scheme.  Starting from the L1
approximation {\cite{LinXu2007}} of the Caputo derivative, the previous state in the MBO
heat step is replaced by a convex L1 history average of all previous
thresholded states.  One then solves a Helmholtz diffusion problem forced by
this history average and thresholds the result.  {For $\alpha=1$, the history average reduces to the most recent
thresholded state and the diffusion operator is the Helmholtz resolvent
$(I-h\Delta)^{-1}$. The resulting method is a resolvent analogue of the
classical MBO algorithm, whose diffusion step uses the heat semigroup
$e^{h\Delta}$.}  For
$0<\alpha<1$, the scheme retains the geometric simplicity of threshold
dynamics while introducing a subdiffusive temporal memory.


Fractional derivatives are nonlocal in time and arise in anomalous
diffusion, trapping and heavy-tailed waiting-time models
\cite{MontrollWeiss1965,ScherMontroll1975,BouchaudGeorges1990,MetzlerKlafter2000}.
Moving-boundary problems with anomalous or subdiffusive transport can behave
quite differently from their integer-order analogues as discussed in \cite{GruberVoglMiksisDavis2013}.
Moreover, the memory considered here is temporal, that is, the diffusion operator remains local in
space and the new difficulty is how to encode past motion of an interface that
itself changes in time.

The present work is also related to recent computational approaches for fractional diffusion and phase-field/interface models. L1-type discretizations of Caputo derivatives remain a standard tool for time-fractional diffusion problems, and recent variants \cite{jiang2024efficient,wang2025numerical} combine such temporal discretizations with efficient spatial solvers and stability or convergence analysis. In parallel, structure-preserving schemes for Allen--Cahn-type phase-field models have been developed in \cite{poochinapan2022numerical,uzunca2023linearly} to retain maximum-principle or energy-dissipation properties at the discrete level. The algorithm introduced here differs from these diffuse-interface approaches: the memory enters through an L1 history average in a threshold-dynamics framework, and the post-diffusion thresholding step enforces pure phases directly.

{We derive the variational characterization of the Helmholtz diffusion
step and establish its stability and scalar maximum/comparison principles.
The convex L1 history representation connects these resolvent properties to
the fractional thresholding construction. Section~\ref{sec:prelim} recalls
the underlying L1 identities and convolution positivity.}

The numerical experiments illustrate the thresholding scheme.  In the equal-tension
multiphase case, each phase indicator is diffused from its own L1 history
average and the new partition is obtained by an argmax rule. {The periodic grain-growth experiments examine coarsening, grain
geometry, and the dependence of the computed evolution on fractional order
and smoothing length.}  

\paragraph{Organization.}
Section~\ref{sec:prelim} discusses the setting and recalls the fractional
{derivative, the L1 weights, and the standard convolution positivity estimate.}
Section~\ref{sec:mbo} defines the history-aware fractional MBO scheme and proves
the variational and comparison properties of its diffusion step.  Furthermore,
the L1 subdiffusion structure behind the thresholding step is explained.
{Section~\ref{sec:structural} examines the interaction between Helmholtz
smoothing and temporal memory through the memory-tail estimate, the pinning
criterion with memory, and the first-step normalization of the Helmholtz map.}
Section~\ref{sec:implementation} contains the multiphase implementation and
numerical grain-growth experiments.

\section{Setting and preliminaries}\label{sec:prelim}

We work on the flat torus
$
\Om := [0,\Lambda)^d \cong \T^d
$
with periodic boundary conditions.
All spatial integrals are with respect to Lebesgue measure.
Unless stated otherwise, functions are vector-valued with values in $\R^N$.

\subsection{Fractional derivative and L1 time discretization}
Fix $\alpha\in(0,1)$.
For sufficiently regular $u:[0,T]\to X$, where $X$ is a Banach space, the Caputo derivative is
\begin{equation}\label{eq:Caputo}
\Caputo u(t)
:=
\frac{1}{\Gammafun(1-\alpha)}
\int_0^t (t-s)^{-\alpha}\,\partial_s u(s)\dd s,
\end{equation}
see, e.g., \cite[Definition 2.9]{li2018some}.
The crucial feature is the nonlocal history dependence:
$\Caputo u(t)$ depends on the entire past $\{u(s):0\le s\le t\}$.

This can be further seen in its time discretization:
Let $t_n:=nh$ for $n=0,\dots,N_T$, where $N_T h=T$.
Then the L1 formula is the classical piecewise-linear-in-time approximation of the Caputo derivative; see, for example, \cite{LinXu2007,JinLazarovZhou2016,LiaoLiZhang2018,fritz2024well,fritz2026time}.
Define the L1 weights
\begin{equation}\label{eq:b-weights}
b_k := (k+1)^{1-\alpha}-k^{1-\alpha},\qquad k\ge 0.
\end{equation}
It can be directly seen that the weights $(b_k)_{k\ge 0}$ satisfy
\begin{equation}\label{eq:b-basic}
b_k>0,\qquad b_{k+1}\le b_k,\qquad b_k\sim (1-\alpha)\,k^{-\alpha}\qquad (k\to\infty).
\end{equation}
All analytical statements involving the L1 weights are stated for
$0<\alpha<1$, unless a limiting case is explicitly mentioned.  When we use
$\alpha=1$, we use the classical limiting convention
\begin{equation*}
b_0=1,\qquad b_k=0\quad(k\ge1).
\end{equation*}
With this convention the history average reduces to
$\bar v^{\,n-1}=v^{n-1}$, and
$
\ell_{1,h}^2=h$.
Thus the case $\alpha=1$ is the implicit-Euler Helmholtz-threshold reference
scheme.
For a sequence $(v^n)_{n\ge 0}$, we define the L1 discrete Caputo operator {as in \cite{LinXu2007}} by
\begin{equation}\label{eq:L1op}
(\partial_h^\alpha v)^n
:=
\frac{1}{\Gammafun(2-\alpha)\,h^\alpha}
\sum_{k=0}^{n-1} b_k\,(v^{n-k}-v^{n-k-1}),
\qquad n\ge 1.
\end{equation}
We also use the increment notation
$\delta v^n := v^n-v^{n-1}$ for $n\ge 1$

{The following L1 history representation is recalled from
\cite[Eqs.~(3.7)--(3.9)]{LinXu2007}.}
\begin{lemma}[{L1 increment and history identities}]\label{lem:L1conv}
For every $n\ge 1$,
\begin{equation}\label{eq:L1-increment-form}
(\partial_h^\alpha v)^n
=
\frac{1}{\Gammafun(2-\alpha)\,h^\alpha}
\sum_{\ell=1}^n b_{n-\ell}\,\delta v^\ell.
\end{equation}
Equivalently, if one defines the history average
\begin{equation}\label{eq:history-average-general}
\bar v^{\,n-1}
:=
\sum_{j=1}^{n-1}(b_{j-1}-b_j)\,v^{n-j}+b_{n-1}v^0,
\end{equation}
then the coefficients in \eqref{eq:history-average-general} are nonnegative, sum to $1$, and
\begin{equation}\label{eq:L1-history-form}
(\partial_h^\alpha v)^n
=
\frac{1}{\Gammafun(2-\alpha)\,h^\alpha}\bigl(v^n-\bar v^{\,n-1}\bigr).
\end{equation}
\end{lemma}


{The following positivity estimate for the L1 convolution holds for
arbitrary discrete histories; see
\cite[Lemma~3.1]{TangYuZhou2019Revision}.}

\begin{lemma}[L1 positivity / discrete fractional integration by parts]
\label{lem:L1pos}
Let $(\mathcal H,\langle\cdot,\cdot\rangle_{\mathcal H})$ be a real Hilbert
space.  For any $m\in\mathbb N$ and any sequence
$(v^n)_{n=0}^m\subset\mathcal H$,
\begin{equation}\label{eq:L1pos-Hilbert}
\sum_{n=1}^m
\left\langle
(\partial_h^\alpha v)^n,
\delta v^n
\right\rangle_{\mathcal H}
\ge
\frac{1}{2\Gammafun(2-\alpha)h^\alpha}
\sum_{n=1}^m
\|\delta v^n\|_{\mathcal H}^2 .
\end{equation}
In particular, with $\mathcal H=L^2(\Omega;\mathbb R^N)$,
\begin{equation}\label{eq:L1pos}
\sum_{n=1}^m \int_\Om (\partial_h^\alpha v)^n \cdot \delta v^n\dd x
\ge
\frac{1}{2\Gammafun(2-\alpha)h^\alpha}
\sum_{n=1}^m \|\delta v^n\|_{L^2(\Om)}^2 .
\end{equation}
\end{lemma}

{The scalar inequality is \cite[Lemma~3.1]{TangYuZhou2019Revision}, after converting their weights
$\beta_j=\frac{h^{1-\alpha}}{\Gammafun(2-\alpha)}b_j$
to the normalization in \eqref{eq:L1op}.
For the Hilbert-space version, apply that scalar inequality to each coordinate
of $\delta v^1,\ldots,\delta v^m$ in an orthonormal basis of their finite-dimensional
span and sum.}

\section{A history-aware fractional MBO scheme}\label{sec:mbo}

We begin with the prototypical subdiffusive equation associated with the time-fractional $L^2$-gradient flow (in the sense of \cite{FritzKhristenkoWohlmuth2023}) of the Dirichlet energy
\begin{equation}\label{eq:subdiff-heat}
\Caputo u = \Delta u.
\end{equation}
In the classical case $\alpha=1$, the MBO scheme is obtained by alternating a heat step and a thresholding step.
For $\alpha\in(0,1)$, the L1 approximation naturally introduces a nontrivial history average in the diffusion step:

Given a sequence of thresholded states $(\chi^k)_{k=0}^{n-1}$, define the L1 history average
\begin{equation}\label{eq:ubar}
\bar\chi^{\,n-1}
:=
\sum_{j=1}^{n-1}(b_{j-1}-b_j)\,\chi^{n-j} + b_{n-1}\,\chi^0.
\end{equation}
By Lemma~\ref{lem:L1conv}, the coefficients are nonnegative and sum to $1$.
Thus $\bar\chi^{\,n-1}$ is a convex combination of the past thresholded states.

\subsection{Fractional Helmholtz-threshold scheme (two-phase)}

\begin{definition}[Fractional Helmholtz-threshold scheme]
Let $\chi^0\in L^\infty(\Om;\{0,1\})$.  For $n\ge 1$, assume that
$\chi^0,\dots,\chi^{n-1}\in L^\infty(\Om;\{0,1\})$ have already been
constructed and define the L1 history average $\bar\chi^{\,n-1}$ by
\eqref{eq:ubar}.  The pre-threshold field $w^n\in H^1(\Om)$ is the weak
solution of
\begin{equation}\label{eq:mbo-diff}
\Big(\frac{1}{\Gammafun(2-\alpha)h^\alpha}I-\Delta\Big)w^n
=
\frac{1}{\Gammafun(2-\alpha)h^\alpha}\,\bar\chi^{\,n-1}.
\end{equation}
Equivalently, it holds
$w^n=(I-\ell_{\alpha,h}^2\Delta)^{-1}\bar\chi^{\,n-1}$
where  
$\ell_{\alpha,h}^2=\Gammafun(2-\alpha)h^\alpha$,
with the periodic boundary condition inherited from $\Om$.  The next thresholded state is
\begin{equation}\label{eq:mbo-threshold}
\chi^n := T(w^n),
\qquad
T(s):=\1_{\{s>1/2\}} .
\end{equation}
Thus points satisfying $w^n=1/2$ are assigned to the zero phase.  This fixed
convention is used throughout the scalar two-phase analysis.  With this choice
the threshold map $T$ is nondecreasing, which is needed in the scalar
comparison argument below.
\end{definition}

Equation \eqref{eq:mbo-diff} is the Helmholtz-type resolvent associated with the L1 discretization of \eqref{eq:subdiff-heat}: by \eqref{eq:L1-history-form}, the implicit L1 scheme for $\Caputo u=\Delta u$ reads
\begin{equation*}
\frac{u^n-\bar u^{\,n-1}}{\Gammafun(2-\alpha)h^\alpha}=\Delta u^n.
\end{equation*}
The notation distinguishes the pre-threshold diffused field $w^n$, which solves \eqref{eq:mbo-diff}, from the post-threshold phase indicator $\chi^n$.

\begin{proposition}[Variational characterization of the diffusion step]\label{prop:var-mbo}
The solution $w^n$ of \eqref{eq:mbo-diff} is the unique minimizer (over $H^1(\Om)$) of
\begin{equation}\label{eq:Jmbo}
\mathcal{J}_n(w)
=
\frac12\int_\Om |\nabla w|^2\dd x
+
\frac{1}{2\Gammafun(2-\alpha)h^\alpha}\int_\Om |w-\bar\chi^{\,n-1}|^2\dd x.
\end{equation}
\end{proposition}

\begin{proof}
The functional $\mathcal{J}_n$ is strictly convex and coercive on $H^1(\Om)$, hence it has a unique minimizer.
Its Euler--Lagrange equation is precisely \eqref{eq:mbo-diff}.
\end{proof}

{The fidelity term in \eqref{eq:Jmbo} is centered on the L1 history
average $\bar\chi^{\,n-1}$, a weighted combination of all past thresholded
states. This dependence on the phase history encodes the discrete memory
mechanism of the Caputo derivative.}

{The history averaging preserves the order structure of the scalar
diffusion step. The following maximum/comparison principle gives bounds for
the pre-threshold fields and transfers their ordering through the scalar
threshold map.}

\begin{proposition}[Maximum and comparison principle for the scalar diffusion step]\label{prop:mbo-maxprinciple}
Let $\chi^0\in L^\infty(\Om;\{0,1\})$, and let $w^n$ be defined by \eqref{eq:mbo-diff}.  Then, for every $n\ge1$,
\begin{equation*}
0\le \bar\chi^{\,n-1}\le 1
\qquad\text{and}\qquad
0\le w^n\le1
\quad\text{a.e. in }\Om.
\end{equation*}
Moreover, if $(\chi^k)_{k=0}^{n-1}$ and $(\widetilde\chi^k)_{k=0}^{n-1}$ are two histories with $\chi^k\le\widetilde\chi^k$ a.e. for all $k\le n-1$, then the corresponding history averages and diffused fields satisfy
\begin{equation*}
\bar\chi^{\,n-1}\le \bar{\widetilde\chi}^{\,n-1},
\qquad
w^n\le \widetilde w^n
\quad\text{a.e. in }\Om.
\end{equation*}
Consequently, thresholding at the fixed level $1/2$ preserves this order pointwise for the two scalar diffused fields.
\end{proposition}

\begin{proof}
By Lemma~\ref{lem:L1conv}, $\bar\chi^{\,n-1}$ is a convex combination of the previous indicators.  Hence $0\le\bar\chi^{\,n-1}\le1$ a.e.  Set $a=(\Gammafun(2-\alpha)h^\alpha)^{-1}>0$.  The diffusion equation reads
\begin{equation*}
(aI-\Delta)w^n=a\bar\chi^{\,n-1}.
\end{equation*}
Testing this equation with $(w^n-1)^+$ gives
\begin{equation*}
a\|(w^n-1)^+\|_{L^2(\Om)}^2
+\|\nabla(w^n-1)^+\|_{L^2(\Om)}^2
\le
 a\int_\Om (\bar\chi^{\,n-1}-1)(w^n-1)^+\dd x\le0,
\end{equation*}
and therefore $w^n\le1$.  Testing with $- (w^n)^-$ analogously gives $w^n\ge0$.

For the comparison statement, the monotonicity of the history average follows again from the nonnegative L1 coefficients.  If $\bar\chi^{\,n-1}\le \bar{\widetilde\chi}^{\,n-1}$, the difference $z=w^n-\widetilde w^n$ satisfies $(aI-\Delta)z\le0$.  Testing with $z^+$ yields $z^+=0$, hence $w^n\le\widetilde w^n$.  The pointwise implication for thresholded indicators follows immediately from the monotonicity of the map $s\mapsto\mathbf 1_{\{s>1/2\}}$.
\end{proof}

\subsection{The L1 subdiffusion structure behind thresholding}\label{sec:subdiffusion}

{We now relate the fractional MBO diffusion step to the L1 subdiffusion
resolvent with a prescribed history.}

Let $(f^k)_{k=0}^{n-1}\subset L^2(\Om)$ be a prescribed history and define
\begin{equation*}
\bar f^{\,n-1}
:=
\sum_{j=1}^{n-1}(b_{j-1}-b_j)f^{n-j}+b_{n-1}f^0 .
\end{equation*}
The implicit L1 step for $\Caputo u=\Delta u$, with this history, is
\begin{equation}\label{eq:l1-subdiffusion-step}
\frac{u^n-\bar f^{\,n-1}}{\Gammafun(2-\alpha)h^\alpha}
=
\Delta u^n .
\end{equation}
Equivalently, it holds
\begin{equation}\label{eq:l1-subdiffusion-resolvent}
\left(\frac{1}{\Gammafun(2-\alpha)h^\alpha}I-\Delta\right)u^n
=
\frac{1}{\Gammafun(2-\alpha)h^\alpha}\bar f^{\,n-1} .
\end{equation}
{Thus the diffusion step of the fractional Helmholtz-threshold scheme
\eqref{eq:mbo-diff} is \eqref{eq:l1-subdiffusion-resolvent} with
$f^k=\chi^k$; applying the threshold map then gives the phase indicators.}

\begin{proposition}[Subdiffusion resolvent properties]\label{prop:subdiffusion-resolvent}
Let $a=(\Gammafun(2-\alpha)h^\alpha)^{-1}$, and let $u\in H^1(\Om)$ solve
\begin{equation}\label{eq:resolvent-af}
(aI-\Delta)u=a f
\end{equation}
with periodic boundary conditions, where $f\in L^2(\Om)$.  Then:
\begin{enumerate}[leftmargin=*]
\item $u$ is the unique minimizer (over $H^1(\Om)$) of
\begin{equation*}
J_f(v)=\frac12\int_\Om |\nabla v|^2\dd x
+\frac a2\int_\Om |v-f|^2\dd x.
\end{equation*}
\item The resolvent preserves the spatial average:
\begin{equation*}
\int_\Om u\dd x=\int_\Om f\dd x .
\end{equation*}
\item It is $L^2$-contractive in the sense that
\begin{equation*}
\|u\|_{L^2(\Om)}\le \|f\|_{L^2(\Om)} .
\end{equation*}
More generally, if $(aI-\Delta)u_i=af_i$ for $i=1,2$, then
\begin{equation*}
\|u_1-u_2\|_{L^2(\Om)}\le \|f_1-f_2\|_{L^2(\Om)} .
\end{equation*}
\item If $0\le f\le1$ a.e., then $0\le u\le1$ a.e.  If $f_1\le f_2$ a.e., then the corresponding solutions satisfy $u_1\le u_2$ a.e.
\end{enumerate}
\end{proposition}

\begin{proof}
The variational characterization follows from strict convexity, coercivity, and
the Euler--Lagrange equation.  Integrating \eqref{eq:resolvent-af} over the
periodic torus gives preservation of the spatial average.  Testing
\eqref{eq:resolvent-af} by $u$ yields
\begin{equation*}
a\|u\|_{L^2}^2+\|\nabla u\|_{L^2}^2=a\int_\Om fu\dd x
\le a\|f\|_{L^2}\|u\|_{L^2},
\end{equation*}
and hence $\|u\|_{L^2}\le\|f\|_{L^2}$.  Applying the same argument to the
difference of two resolvent equations gives the contraction estimate.

For the maximum principle, test \eqref{eq:resolvent-af} by $(u-1)^+$.  If
$f\le1$, then
\begin{equation*}
a\|(u-1)^+\|_{L^2}^2+\|\nabla (u-1)^+\|_{L^2}^2
\le a\int_\Om (f-1)(u-1)^+\dd x\le0,
\end{equation*}
so $u\le1$.  Testing by $-u^-$ gives $u\ge0$ when $f\ge0$.  The
comparison principle follows by applying the same argument to
$(u_1-u_2)^+$.
\end{proof}

\begin{remark}[Mass and thresholding]
{The Helmholtz resolvent preserves the spatial average of its input.
Thresholding allows phase volumes to change, as in ordinary MBO motion by
mean curvature. Volume-preserving variants introduce an additional volume
constraint or a modified threshold rule.}
\end{remark}

\section{Structural properties and the role of memory}\label{sec:structural}

{The scheme combines the structural properties of the Helmholtz-threshold map}
\begin{equation*}
  f\longmapsto \1_{\{(I-\ell^2\Delta)^{-1}f>1/2\}},
\end{equation*}
{with the temporal dependence of its input. For a fixed input, the
fractional order enters this map through the smoothing length}
\begin{equation}\label{eq:ell-def}
  \ell_{\alpha,h}^2:=\Gammafun(2-\alpha)h^\alpha .
\end{equation}
{Temporal memory enters through the input datum}
\(f=\bar\chi^{\,n-1}\), the L1 history average of previous thresholded states.
{The memory-tail estimate and the pinning criterion with memory quantify
this history dependence.}  With \eqref{eq:ell-def}, the scalar diffusion step is
\begin{equation}\label{eq:resolvent-ell}
  w^n=(I-\ell_{\alpha,h}^2\Delta)^{-1}\bar\chi^{\,n-1}.
\end{equation}

\subsection{Projection and resolvent structure}

The thresholding step is the exact pointwise projection of the diffused field
onto the pure phases.  In the scalar case, for \(w\in L^2(\Om)\), every
minimizer of
\begin{equation*}
  \min_{\eta\in L^2(\Om;\{0,1\})}\int_\Om |\eta-w|^2\dd x
\end{equation*}
satisfies
\begin{equation*}
  \eta=1 \quad\text{a.e. on }\{w>1/2\},
  \qquad
  \eta=0 \quad\text{a.e. on }\{w<1/2\},
\end{equation*}
with arbitrary measurable choices on \(\{w=1/2\}\).  In the multiphase case,
projection onto \(\{e_1,\ldots,e_P\}\subset\R^P\) is obtained pointwise by
choosing an index in \(\operatorname*{arg\,max}_j w_j(x)\).  This follows from
\begin{equation*}
  |1-w|^2-|w|^2=1-2w,
  \qquad
  |e_i-w|^2=1-2w_i+|w|^2,
\end{equation*}
and is the variational interpretation of the scalar threshold and of the
minimum-index argmax rule \eqref{eq:num-argmax}.

The Helmholtz step also has the heat-mixture representation
\begin{equation}\label{eq:resolvent-heat-mixture-torus}
  (I-\ell^2\Delta)^{-1}f
  =
  \int_0^\infty e^{-s}e^{s\ell^2\Delta}f\dd s
\end{equation}
in \(L^2(\Om)\).  Hence it preserves constants, is positivity preserving, and is
an \(L^2\)-contraction.  On \(\R^d\), for bounded functions we use
\begin{equation}\label{eq:Rell-Rd}
  \mathcal R_\ell^{\R^d}f(x)
  :=
  \int_0^\infty e^{-s}(G_{s\ell^2}*f)(x)\dd s,
\end{equation}
where \(G_t\) is the heat kernel.  This operator is positivity preserving,
preserves constants, is an \(L^\infty\)-contraction, and satisfies
\((I-\ell^2\Delta)\mathcal R_\ell^{\R^d}f=f\) in distributions.  These facts
follow directly from the Fourier multiplier
\((1+\ell^2\lambda_k)^{-1}\) on the torus and from the nonnegative heat-kernel
average \eqref{eq:Rell-Rd} on \(\R^d\).

\begin{remark}[Effective length scale]\label{rem:effective-length-scale}
The parameter
$\ell_{\alpha,h}=(\Gammafun(2-\alpha)h^\alpha)^{1/2}$
is the spatial smoothing length of the Helmholtz resolvent.  At fixed \(h\),
changing \(\alpha\) changes both the L1 history weights and the one-step
smoothing length.  The monotonicity of \(\ell_{\alpha,h}\) as a function of
\(\alpha\) depends on \(h\), since
\begin{equation*}
  \frac{\dd}{\dd\alpha}\log \ell_{\alpha,h}^2
  =
  \log h-\psi(2-\alpha),
\end{equation*}
where \(\psi=\Gamma'/\Gamma\).  For the small time step used in the computations
below, \(h=5\cdot10^{-4}\), the displayed values of \(\alpha\) lie in a regime
where smaller \(\alpha\) gives a larger smoothing length.
\end{remark}

\subsection{Where the fractional parameter enters}
\label{subsec:where-alpha-enters}

{We now examine how the fractional order determines the history average
supplied to the Helmholtz-threshold map.} By \eqref{eq:history-average-general}, it holds
\begin{equation}\label{eq:history-alpha-weights}
\bar\chi^{\,n-1}
=
\sum_{r=0}^{n-2}\omega_r^{(\alpha)}\chi^{n-1-r}
+
\omega_{n-1}^{(\alpha)}\chi^0,
\end{equation}
where the weights are defined by
\begin{equation}\label{eq:omega-alpha}
\omega_r^{(\alpha)}:=b_r-b_{r+1}\quad(0\le r\le n-2),
\qquad
\omega_{n-1}^{(\alpha)}:=b_{n-1}.
\end{equation}
We note that the weights are nonnegative and sum to one.
For scalar phase indicators, or coordinate-vector indicators with the
componentwise $L^\infty$ norm, the limit as $\alpha\uparrow1$
is in fact uniform in the history length. For $n\ge2$, the total coefficient
of all states other than $\chi^{n-1}$ is $b_1$, so
\begin{equation}\label{eq:uniform-alpha-one-history}
 \|\bar\chi^{\,n-1}-\chi^{n-1}\|_{L^\infty}
 \le b_1=2^{1-\alpha}-1\longrightarrow0.
\end{equation}
For $n=1$ the difference is zero. More generally the right-hand side is
$D b_1$ for histories of diameter at most $D$. This estimate is uniform over prescribed bounded histories. Extending it
to evolving thresholded trajectories requires control of the pre-threshold
fields near the threshold level, where the threshold map is discontinuous.

\begin{lemma}[Memory-tail estimate]\label{lem:memory-tail}
Let \(0<\alpha<1\), and let \((z^k)_{k=0}^{n-1}\) be a sequence with values in a
normed vector space.  Assume that the set of values has diameter at most \(D\).
Define
\begin{equation*}
\bar z^{\,n-1}
=
\sum_{r=0}^{n-2}\omega_r^{(\alpha)}z^{n-1-r}
+
\omega_{n-1}^{(\alpha)}z^0 .
\end{equation*}
Then
\begin{equation}\label{eq:memory-tail-general}
  \|\bar z^{\,n-1}-z^{n-1}\|
  \le
  \sum_{r=1}^{n-1}\omega_r^{(\alpha)}
  \|z^{n-1-r}-z^{n-1}\|.
\end{equation}
If, in addition,
\(z^{n-1}=z^{n-2}=\cdots=z^{n-q}\) for some \(1\le q\le n-1\), then
\begin{equation}\label{eq:memory-tail-bq}
  \|\bar z^{\,n-1}-z^{n-1}\|
  \le
  D b_q
  =
  D\bigl((q+1)^{1-\alpha}-q^{1-\alpha}\bigr).
\end{equation}
For scalar phase indicators in \(L^\infty\), or coordinate-vector phase
indicators in the pointwise \(\ell^\infty\)-norm, one may take \(D=1\).
\end{lemma}

\begin{proof}
Since the weights sum to one,
\begin{equation*}
  \bar z^{\,n-1}-z^{n-1}
  =
  \sum_{r=1}^{n-1}\omega_r^{(\alpha)}(z^{n-1-r}-z^{n-1}),
\end{equation*}
which gives \eqref{eq:memory-tail-general}.  If the last \(q\) previous states
agree with \(z^{n-1}\), then the summands with \(1\le r\le q-1\) vanish, and
\begin{equation*}
  \sum_{r=q}^{n-1}\omega_r^{(\alpha)}
  =
  \sum_{r=q}^{n-2}(b_r-b_{r+1})+b_{n-1}
  =b_q .
\end{equation*}
This proves \eqref{eq:memory-tail-bq}.
\end{proof}

\begin{remark}[Interpretation]
The memory tail satisfies \(b_q\sim(1-\alpha)q^{-\alpha}\) as \(q\to\infty\).
Thus, even if an interface has been locally unchanged for many steps, older
thresholded states can still bias the next pre-diffusion datum.  This is the
part of the algorithm where the fractional order affects the dynamics beyond the
single-step smoothing length \(\ell_{\alpha,h}\).
\end{remark}

\begin{remark}[{Small-order limit at a fixed time step}]\label{rem:small-alpha}
Fix the dimensionless time step $h>0$ and a finite time level $n$.
As $\alpha\downarrow0$,
\begin{equation}\label{eq:small-alpha-parameters}
 b_k\longrightarrow1\quad\text{for each fixed }k,\qquad
 \ell_{\alpha,h}^2\longrightarrow1.
\end{equation}
For scalar phase-indicator histories with the same initial state (or
componentwise for coordinate-vector indicators), even if the later
states depend on $\alpha$, convexity gives
\begin{equation}\label{eq:small-alpha-history}
 \|\bar\chi^{\,n-1}-\chi^0\|_{L^\infty}
 \le 1-b_{n-1}\longrightarrow0.
\end{equation}
Indeed, $1-b_{n-1}$ is exactly the total coefficient of the states with
positive time indices. Consequently, the continuous pre-threshold field satisfies
\begin{equation*}
 w^n\longrightarrow (I-\Delta)^{-1}\chi^0\qquad\text{in }L^2(\Om).
\end{equation*}
This follows from \eqref{eq:small-alpha-history}, resolvent contraction, and
convergence of the Fourier multipliers $(1+\ell_{\alpha,h}^2\lambda_k)^{-1}$.
Convergence of the corresponding thresholded states requires additional
control on any limiting tie set. For small positive orders, the initial
state therefore dominates the finite-level history, while the smoothing
length is comparable to the side length of the unit computational torus.

\end{remark}

\subsection{Flat interfaces and finite-dimensional pinning}

An exactly flat interface is fixed by the Helmholtz-threshold map.  Indeed, if
\(H=\{x\in\R^d:x\cdot\nu>c\}\), \(|\nu|=1\), and \(\chi=\1_H\), then
\begin{equation}\label{eq:flat-interface-values}
  w:=\mathcal R_\ell^{\R^d}\chi
\end{equation}
satisfies
\begin{equation*}
  w(x)>\frac12\quad\text{if }x\cdot\nu>c,
  \qquad
  w(x)=\frac12\quad\text{if }x\cdot\nu=c,
  \qquad
  w(x)<\frac12\quad\text{if }x\cdot\nu<c.
\end{equation*}
This follows from \eqref{eq:Rell-Rd}: each heat convolution with a radial kernel
has the corresponding half-space symmetry and monotonicity, and the exponential
average preserves these inequalities.  Hence thresholding recovers the same
half-space, up to the deterministic convention on the boundary hyperplane.  The
same symmetry argument applies to periodic flat slabs when the complementary
phases are arranged symmetrically with respect to the relevant flat interfaces.

{We also record a finite-dimensional pinning criterion.}  Let
\(L\in\R^{M\times M}\) be a graph or finite-difference Laplacian with
\(L\mathbf1=0\), and assume that
\begin{equation}\label{eq:linfty-resolvent-contraction}
  \|(I+\ell^2L)^{-1}g\|_{\ell^\infty}
  \le
  \|g\|_{\ell^\infty}
  \qquad\text{for all }g\in\R^M .
\end{equation}
This holds, for instance, under the usual monotone \(M\)-matrix hypotheses on
\(I+\ell^2L\).  If \(\chi\in\{0,1\}^M\), \(f\in[0,1]^M\), and
\(w=(I+\ell^2L)^{-1}f\), then
\begin{equation}\label{eq:pinning-condition}
  \|f-\chi\|_{\ell^\infty}+\ell^2\|L\chi\|_{\ell^\infty}<\frac12
\end{equation}
implies $\1_{\{w_i>1/2\}}=\chi_i$ for any $i=1,\ldots,M$
Indeed, it holds
\begin{equation*}
  w-\chi=(I+\ell^2L)^{-1}(f-\chi-\ell^2L\chi),
\end{equation*}
and \eqref{eq:linfty-resolvent-contraction} gives
\(\|w-\chi\|_{\ell^\infty}<1/2\).  At the first step, where \(f=\chi\), the
sufficient pinning condition reduces to
\begin{equation}\label{eq:first-step-pinning}
  \ell^2\|L\chi\|_{\ell^\infty}<\frac12 .
\end{equation}

Combining \eqref{eq:pinning-condition} with the memory-tail estimate gives the
fractional version: if \(f=\bar\chi^{\,n-1}\), \(\chi=\chi^{n-1}\), and
\(\chi^{n-1}=\chi^{n-2}=\cdots=\chi^{n-q}\), then the next thresholding step is
pinned whenever
\begin{equation}\label{eq:pinning-memory-tail}
  b_q+
  \ell_{\alpha,h}^2\|L\chi^{n-1}\|_{\ell^\infty}<\frac12 .
\end{equation}
For a nearest-neighbour grid Laplacian, \(\|L\chi\|_{\ell^\infty}\) is of order
\(\Delta x^{-2}\) at grid points adjacent to an interface, so
\eqref{eq:first-step-pinning} expresses the familiar finite-grid rule that the
smoothing length must be comparable to the grid scale before thresholding can
move an interface.  {The criterion uses the resolvent contraction assumption
\eqref{eq:linfty-resolvent-contraction}. Its application to the finite-element
discretization is discussed in Section~\ref{sec:implementation}.}

\subsection{First-step normalization and formal sharp-interface heuristic}

{At the first step, \(\bar\chi^0=\chi^0\), so the Helmholtz-threshold
map acts directly on the initial indicator.}
For a compact \(C^4\) set \(E\subset\R^d\), with signed distance \(d_E>0\) in
\(E\), inward normal \(\nu=\nabla d_E\), and
\(\kappa=-\Delta d_E|_{\partial E}\), the usual local boundary-layer expansion
for the Helmholtz resolvent gives, for \(w_\ell=\mathcal R_\ell^{\R^d}\1_E\),
\begin{equation}\label{eq:first-step-local-expansion-short}
  w_\ell(y+\ell^2\eta\nu(y))
  =
  \frac12+
  \ell\left(\frac\eta2-\frac{\kappa(y)}4\right)+O_R(\ell^2)
\end{equation}
uniformly for \(y\in\partial E\) and \(|\eta|\le R\), and
\(\partial_\nu w_\ell(y+\ell^2\eta\nu(y))=(2\ell)^{-1}+O_R(1)\).  {The expansion follows the local analysis of convolution-generated
threshold dynamics;} see, for
example, \cite{IshiiPiresSouganidis1999} for general threshold-dynamics
approximation schemes and \cite{EsedogluJacobsZhang2017} for the role of
convolution kernels in prescribing surface tension and mobility.  The displayed
expansion fixes the normalization of the present Helmholtz kernel.  It implies
that the new \(1/2\)-level set is a normal graph
\begin{equation}\label{eq:first-step-displacement}
  \rho_\ell(y)=\frac{\ell^2}{2}\kappa(y)+O(\ell^3).
\end{equation}
Thus the first Helmholtz-threshold step has the mean-curvature displacement
corresponding to effective time \(\ell^2/2\), with the sign convention
\(d_E>0\) in the phase.  

Formally, if all relevant past interfaces are smooth and can be transported as
normal graphs over the current interface on the scale \(O(\ell_{\alpha,h}^2)\),
linearizing \eqref{eq:first-step-local-expansion-short} for each indicator in
the L1 history average suggests
\begin{equation*}
  \bigl(X^n-\bar X^{\,n-1}\bigr)\cdot\nu^{n-1}
  \approx
  \frac{\ell_{\alpha,h}^2}{2}\kappa^{n-1},
\end{equation*}
or, equivalently,
$(\partial_h^\alpha X)^n\cdot\nu^{n-1}\approx \frac12\kappa^{n-1}$.
{This relation is a formal geometric interpretation under the smoothness
and transport assumptions above. The Caputo derivative of an interface
parametrization depends on how past interfaces are transported. A rigorous
sharp-interface analysis therefore requires compactness of the thresholded
interfaces and identification of a canonical or otherwise controlled memory term.}


\section{Numerical experiments}\label{sec:implementation}

This section reports the numerical experiments used in the figures.  All
computations in this section use the history-aware fractional thresholding
scheme described in \Cref{sec:mbo,sec:subdiffusion}. 

\subsection{Discrete multiphase scheme used in the computations}

We use the equal-surface-tension multiphase version of the MBO step, in the spirit of multiphase threshold dynamics and diffusion-generated grain-growth algorithms \cite{MerrimanBenceOsher1994,EsedogluOtto2015,ElseyEsedogluSmereka2009}.  Let
$\chi_k^n:\Om\to\{0,1\}$, $k=1,\dots,P$, denote the phase indicators at time $t_n=nh$, with
$\sum_{k=1}^P \chi_k^n=1$ a.e. in $\Om$.
For $\alpha\in(0,1)$, the diffused field $w_k^n$ is obtained from
\begin{equation}\label{eq:num-fractional-diffusion}
\left(\frac{1}{\Gammafun(2-\alpha)h^\alpha}I-\Delta\right)w_k^n
=
\frac{1}{\Gammafun(2-\alpha)h^\alpha}\,\bar\chi_k^{\,n-1},
\end{equation}
where
\begin{equation}\label{eq:num-history}
\bar\chi_k^{\,n-1}
=
\sum_{j=1}^{n-1}(b_{j-1}-b_j)\chi_k^{n-j}+b_{n-1}\chi_k^0,
\qquad
b_j=(j+1)^{1-\alpha}-j^{1-\alpha}.
\end{equation}
The history average preserves the partition constraint.  Indeed, since the L1
coefficients are nonnegative and sum to one, it implies
$\bar\chi_k^{\,n-1}\ge0$ and
$\sum_{k=1}^P\bar\chi_k^{\,n-1}=1$
a.e. in $\Om$,
whenever it holds $\sum_{k=1}^P\chi_k^m=1$ a.e.~for all previous time levels.  At the
continuous level the Helmholtz resolvent preserves constants; summing
\eqref{eq:num-fractional-diffusion} over $k$ gives
\begin{equation*}
\left(\frac{1}{\Gammafun(2-\alpha)h^\alpha}I-\Delta\right)
\sum_{k=1}^P w_k^n
=
\frac{1}{\Gammafun(2-\alpha)h^\alpha},
\end{equation*}
and uniqueness of the periodic Helmholtz problem implies
$\sum_{k=1}^P w_k^n=1$ a.e. in $\Om$.
The minimum-index argmax rule (or redistribution rule) \begin{equation}\label{eq:num-argmax}
\chi_k^n(x)=1
\quad\Longleftrightarrow\quad
k=\min\operatorname*{arg\,max}_{\ell=1,\dots,P}w_\ell^n(x)
\end{equation} therefore produces a
measurable partition.  The same conclusion holds for a spatial discretization
provided the discrete Helmholtz solver preserves constants.

The minimum-index convention is used only to make ties deterministic.  For $\alpha=1$, the L1 weights are replaced by their classical limit, $b_0=1$ and $b_j=0$ for $j\ge1$, so that $\bar\chi_k^{\,n-1}=\chi_k^{n-1}$.  The resulting method is the equal-tension multiphase Helmholtz-threshold scheme
\begin{equation*}
  (I-h\Delta)w_k^n=\chi_k^{n-1},
  \qquad
  \chi_k^n(x)=1
  \Longleftrightarrow
  k=\min\operatorname*{arg\,max}_{\ell=1,\dots,P}w_\ell^n(x).
\end{equation*}
It is a resolvent analogue of the usual heat-kernel multiphase
MBO scheme.

\subsection{Implementation details}

The computations were carried out on the flat unit torus, represented by a periodic finite-element mesh of the unit square.  The Helmholtz problems in \eqref{eq:num-fractional-diffusion} are discretized with continuous piecewise affine finite elements on a structured $200\times200$ periodic mesh.  At every time step one scalar Helmholtz problem is solved for each label, followed by the pointwise redistribution rule \eqref{eq:num-argmax}.  
{The pinning criterion in \Cref{sec:structural} applies to discretizations
satisfying the resolvent contraction assumption
\eqref{eq:linfty-resolvent-contraction}. For finite elements, mass lumping
together with an $M$-matrix stiffness structure provides a sufficient setting.
Application of the criterion to the consistent-mass $P^1$ scheme used here
remains conditional on verification of that discrete resolvent contraction.}

The simulation parameters are listed in \Cref{tab:numerical-parameters}.
{We consider the fractional parameters
$\alpha\in\{1.0,0.9,0.7,0.5,0.3,0.1\}$.}
The polycrystal simulations use $P=60$ initial grains, time step $h=5\cdot10^{-4}$, and $300$ time steps, so the final time is $T=0.15$.  The random Voronoi seeds are 42, 52, 62, 72, and 82.
For comparisons at equal physical time we use the same initial Voronoi tessellation for all values of $\alpha$, with seed $42$ in the displayed snapshots.  For averaged coarsening curves, the mean and one-standard-deviation bands are computed over the five seeds.  The single-grain benchmark uses two phases, the same grid and time step, and an initially circular central inclusion of radius $0.18$.
The additional polycrystal runs use the same five seeds, mesh, step size and
$300$-step horizon. All phase labels and their complete histories are retained,
and runs continue to the prescribed final step without stopping at a detected
quiescent window. The postprocessing assigns a label to each grid cell by
majority vote over its four vertices, with ties resolved by the smallest
label. Phase areas count these cells, $N_{\rm gr}$ counts their positive-area
labels, and $L$ counts the associated periodic grid-edge interfaces.
A quiescent-window start is the first step $q\ge1$ for which the changed-cell
fraction is zero at each of $q,q+1,\ldots,q+4$; its reported time is $t_q=qh$.

\begin{table}[htp!]
\centering
\caption{Parameters used in the periodic finite-element experiments.}
\label{tab:numerical-parameters}
\begin{tabular}{ll}
\toprule
quantity & value \\
\midrule
computational domain & flat unit torus $\T^2$ \\
spatial grid & periodic $200\times200$ unit-square mesh \\
time step & $h=5\cdot10^{-4}$ \\
polycrystal time steps & $300$, so $T=0.15$ \\
polycrystal initial grains & $P=60$ \\
{fractional parameter} & {$\alpha\in\{1.0,0.9,0.7,0.5,0.3,0.1\}$} \\
polycrystal seeds & $42,52,62,72,82$ \\
displayed comparison seed & $42$ \\
{snapshot selection targets} & $N_{\rm gr}\approx50,40,30$ for \Cref{fig:grain-snapshots-matched-active} \\
{statistical selection target} & $N_{\rm gr}\approx30$ for \Cref{fig:grain-normalized-area-hist,fig:grain-conditional-mean-area,fig:grain-vn-fractional} \\
redistribution & minimum-index $\arg\max$ tie-breaking \\
postprocessing & periodic cellwise label grid \\
\bottomrule
\end{tabular}

\end{table}


\subsection{Single-grain benchmark}

The first test evolves a circular grain embedded in a surrounding matrix on the periodic grid.  {This benchmark measures how the fractional order affects grain shrinkage at a fixed time step.}

\Cref{fig:grain-single-grain} reports the computed area of the central grain
for $\alpha\in\{1,0.9,0.7,0.5\}$. The initial continuum
area is $\pi(0.18)^2\simeq0.101788$, while all four exported initial
cell-area values are $A^0=0.0998$. The first recorded zero cell
area occurs at $t=0.028$, $0.0155$, $0.004$, and $0.001$ for
$\alpha=1,0.9,0.7,0.5$, respectively. These are discrete first-zero
times at the fixed sampling step.
At the fixed step $h=5\cdot10^{-4}$, the corresponding smoothing lengths are
approximately $0.022361$, $0.031894$, $0.066244$, and $0.140772$, respectively.
The faster area loss observed for smaller $\alpha$ accompanies an increase
in smoothing length, together with changes in the L1 history weights.
At the first step, $\bar\chi^0=\chi^0$, so its dependence on $\alpha$
is determined by $\ell_{\alpha,h}$.
The local first-step formula \eqref{eq:first-step-displacement} gives, for a
fixed Euclidean disk in $\R^2$ as $\ell/R_0\to0$,
\begin{equation*}
 R_0-R_1=\frac{\ell^2}{2R_0}+O(\ell^3).
\end{equation*}
The expansion applies in the continuum regime $\ell/R_0\to0$.
For the $\alpha=0.5$ run, $\ell/R_0\simeq0.782$, so the smoothing length
is comparable to the initial radius.

\begin{figure}[htp!]
\centering
\includegraphics[page=1,width=.49\textwidth]{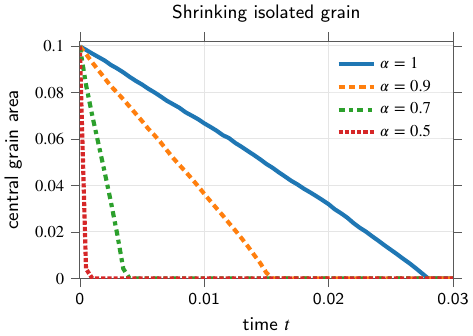}
\caption{{Shrinking isolated-grain benchmark on the periodic $200\times200$
grid, with initial radius $R_0=0.18$ and fixed time step $h=5\cdot10^{-4}$.
The curves report computed inclusion areas for $\alpha\in\{1,0.9,0.7,0.5\}$. All four initial cell areas are
$0.0998$. Both the L1 history and the Helmholtz smoothing length vary across curves.}}
\label{fig:grain-single-grain}
\end{figure}

\subsection{Structural diagnostics for the thresholding step}
\label{subsec:numerics-structural}

The following diagnostics are designed to separate the two places where the
fractional parameter enters the scheme.  Once the pre-diffusion datum $f$ and
the length scale $\ell_{\alpha,h}$ are fixed, the update is the
Helmholtz-threshold map
$f\mapsto
\mathbf 1_{\{(I-\ell_{\alpha,h}^2\Delta)^{-1}f>1/2\}}$.
{The temporal dependence is carried by the L1 history datum}
$f=\bar\chi^{\,n-1}$.

\paragraph{Flat-interface invariance, grid pinning and first-step curvature scaling}
{The test uses the stripe $0.25<x<0.75$, assigned at the grid vertices
with strict inequalities. This nodal assignment treats the boundary vertices
asymmetrically under phase exchange. The left panel of
\Cref{fig:struct-flat-pinning} reports zero changed-cell fraction for the
displayed computation. In the continuum, a half-period stripe is invariant
by reflection symmetry between its complementary phases. For periodic slabs
of other widths, resolvent interactions between the interfaces can affect
invariance. For an unchanged history, the L1 average equals the same
indicator for every fractional order.}

The middle panel of \Cref{fig:struct-flat-pinning} shows a one-step pinning
sweep. {The sweep uses the initial indicator, $\bar\chi^0=\chi^0$, as its
pre-diffusion datum.}  The plot shows the
finite-grid threshold at which the diffuse field begins to cross the level
$1/2$.

The disk test on the right of \Cref{fig:struct-flat-pinning} is another first-step
diagnostic.  Since
the first pre-diffusion datum is exactly the initial indicator, the test probes
the Helmholtz-threshold map.  For a disk of radius $R_0$, the relevant
curvature scale is $1/R_0$, and the measured radius loss is expected to scale
with $\ell_{\alpha,h}^2/R_0$.

\begin{figure}[htp!]
\centering
\includegraphics[page=2,width=.99\textwidth]{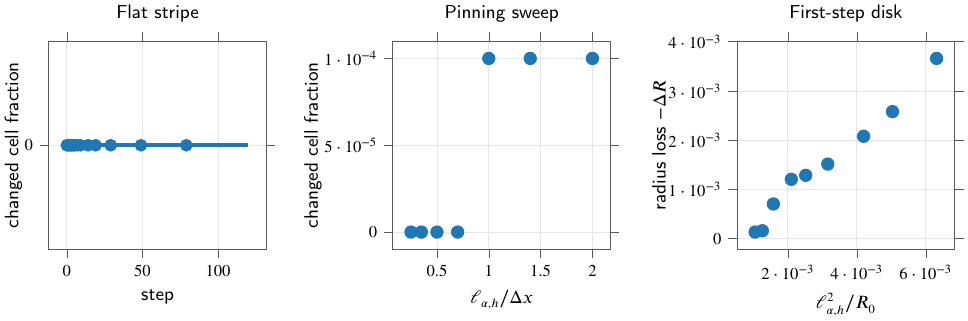}
\caption{Structural diagnostics for flat interfaces and grid pinning. {The displayed $\alpha=1$ data represent the reported coincident curves across the tested orders.}
{Left: zero changed-cell fraction for the periodic stripe computation.} Middle: one-step pinning sweep
plotted against $\ell_{\alpha,h}/\Delta x$.  Right: first-step disk diagnostic;
the measured radius loss is plotted against $\ell_{\alpha,h}^2/R_0$.}
\label{fig:struct-flat-pinning}
\end{figure}

\paragraph{Memory-tail diagnostic.}
The test in the left plot of \Cref{fig:struct-memory-tail} uses $n=q+1$, an initial disk
$\chi^0=\chi_{\rm old}$ of radius $0.28$, and $q$ copies of a concentric disk
$\chi^1=\cdots=\chi^q=\chi_{\rm cur}$ of radius $0.18$.
Consequently the history identity gives exactly
\begin{equation*}
 \bar\chi^{\,q}-\chi_{\rm cur}
 =b_q(\chi_{\rm old}-\chi_{\rm cur}),\qquad
 b_q=(q+1)^{1-\alpha}-q^{1-\alpha}.
\end{equation*}
Thus the mean absolute cellwise error is $b_q$ times the mean absolute
difference between the two disk indicators. This constructed history
attains the corresponding memory-tail bound; for a general older history,
\Cref{lem:memory-tail} supplies an upper bound. Smaller orders give a more
persistent older-state contribution in this diagnostic.

\begin{figure}[htp!]
\centering
\includegraphics[page=3,width=.99\textwidth]{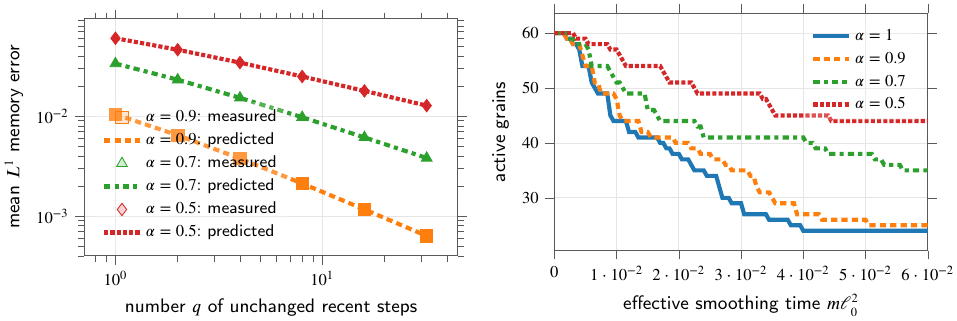}
\caption{{Left: Memory-tail diagnostic for the history consisting of one older
larger disk and $q$ copies of the current disk, with $n=q+1$.
Markers show the measured mean absolute history error; lines show $b_q$
times the mean absolute difference of the two indicators. These coincide
for the constructed history. This test directly probes the L1 history average.} {Right: Polycrystal evolution at a common Helmholtz smoothing length
$\ell_{\alpha,h}=\ell_0$. The curves compare the influence of the L1 history
on successive thresholding updates. The order-dependent time steps
$h_\alpha$ define different physical time grids.}}
\label{fig:struct-memory-tail}
\end{figure}

\paragraph{Fixed-smoothing-length comparison.}
The fixed-$\ell$ polycrystal experiment in the right plot of \Cref{fig:struct-memory-tail}
extends the same idea to long-time dynamics.  Here the time step is chosen as
\begin{equation*}
h_\alpha
=
\left(\frac{\ell_0^2}{\Gammafun(2-\alpha)}\right)^{1/\alpha},
\qquad
\Gammafun(2-\alpha)h_\alpha^\alpha=\ell_0^2,
\end{equation*}
{so that every value of $\alpha$ uses the same Helmholtz smoothing length
at each step. Comparing the evolution over successive updates then examines
the dependence on the L1 history weights. The chosen $h_\alpha$ values define
different physical time grids for these comparisons.}

\subsection{Polycrystalline coarsening}

The main experiment starts from random periodic Voronoi data with $60$ initial grains, a standard initial ensemble for statistical grain-growth computations \cite{AndersonSrolovitzGrestSahni1984,SrolovitzAndersonSahniGrest1984,ElseyEsedogluSmereka2011}.  The same periodic Voronoi tessellation is used for all values of $\alpha$ when comparing snapshots at equal physical time.

The averaged curves in \Cref{fig:grain-coarsening-curves} show the number of
active grains, the mean grain area, and the grid-interface length.  In the
fixed-step experiment reported here, decreasing $\alpha$ leads to faster
coarsening: grains disappear earlier, the mean area increases more rapidly, and
{the total grid-interface length declines more rapidly overall. Small
stepwise increases in the recorded grid-interface length occur for each of
the original four orders.} This trend
is consistent with the effective smoothing scale
$\ell_{\alpha,h}=(\Gammafun(2-\alpha)h^\alpha)^{1/2}$.
For the time step used here, $h=5\cdot10^{-4}$, the displayed values of
$\alpha$ lie in a regime where smaller $\alpha$ gives a larger
$\ell_{\alpha,h}$.  Thus, at fixed $h$, changing $\alpha$ changes both the
history weights and the Helmholtz smoothing length.

\begin{figure}[htp!]
\centering
\includegraphics[page=4,width=.99\textwidth]{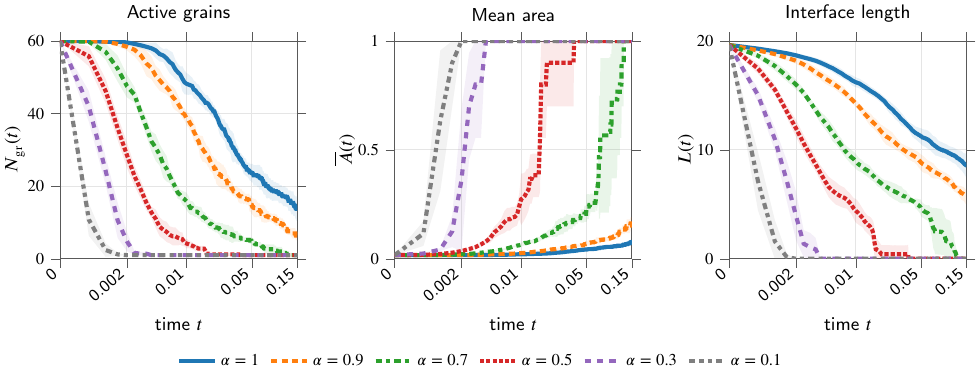}
\caption{{Coarsening diagnostics for
$\alpha\in\{1.0,0.9,0.7,0.5,0.3,0.1\}$, averaged over the same five Voronoi
seeds at fixed step $h=5\cdot10^{-4}$. Left: active-grain count; middle:
mean grain area; right: grid-interface length. Shaded bands indicate one population standard deviation
across seeds, clipped to the displayed ranges. Time is displayed
on the common shifted logarithmic coordinate $\log(1+t/h)$, with ticks
labelled by physical time. This retains $t=0$ and resolves the rapid
small-order transients within the full interval $[0,0.15]$.}}
\label{fig:grain-coarsening-curves}
\end{figure}

The final-state summary in \Cref{tab:steady-summary} quantifies the same effect at $T=0.15$.  The classical reference $\alpha=1$ still has on average $13.8$ active grains, while $\alpha=0.9$ has on average $6.4$.
{For $\alpha\in\{0.7,0.5,0.3,0.1\}$, all five reported final cell partitions
contain one active grain. The mean first quiescent-window start decreases
across these tested orders in the present fixed-step comparison.
For $\alpha=1$, seed $52$ has a first quiescent
window starting at step $278$ ($t=0.139$), although its final cell partition
still has $13$ grains. A quiescent window can therefore occur in a multigrain configuration.}  

\begin{table}[htp!]
\centering
\caption{{Final-state simulation diagnostics at $T=0.15$, averaged over
five Voronoi seeds for each order. The last column gives the mean start time of the first
five-step window of zero cell-label changes, averaged over the runs in
which such a window is detected.}}
\label{tab:steady-summary}
{Reported grain-growth simulations}\par\smallskip
\begin{tabular}{ccccc}
\toprule
$\alpha$ & mean final $N_{\rm gr}$ & mean final $L$ & quiescent runs & mean first quiescent time \\
\midrule
$1.0$ & $13.8$ & $8.580$ & {$0/5$} & -- \\
$0.9$ & $6.4$ & $5.703$ & $0/5$ & -- \\
$0.7$ & $1.0$ & $0.000$ & $5/5$ & $0.0945$ \\
$0.5$ & $1.0$ & $0.000$ & $5/5$ & $0.0215$ \\
{$0.3$} & {$1.0$} & {$0.000$} & {$5/5$} & {$0.0036$} \\
{$0.1$} & {$1.0$} & {$0.000$} & {$5/5$} & {$0.0019$} \\
\bottomrule
\end{tabular}
\end{table}

\begin{table}[htp!]
\centering
\caption{{Parameters calculated
at $h=5\cdot10^{-4}$ and $\Delta x=1/200$. At step $300$, $b_{299}$ is the
initial-state coefficient and $1-b_1$ the most-recent-state coefficient.}}
\label{tab:steady-summary2}
Parameters calculated from the L1 and resolvent formulas\par\smallskip
\begin{tabular}{ccccc}
\toprule
$\alpha$ & $\ell_{\alpha,h}$ & $\ell_{\alpha,h}/\Delta x$ & $b_{299}$ & $1-b_1$ \\
\midrule
$1.0$ & $0.022361$ & $4.472$ & $0.000000$ & $1.000000$ \\
$0.9$ & $0.031894$ & $6.379$ & $0.000591$ & $0.928227$ \\
$0.7$ & $0.066244$ & $13.249$ & $0.005542$ & $0.768856$ \\
$0.5$ & $0.140772$ & $28.154$ & $0.028892$ & $0.585786$ \\
$0.3$ & $0.304818$ & $60.964$ & $0.126526$ & $0.375495$ \\
$0.1$ & $0.670630$ & $134.126$ & $0.508865$ & $0.133934$ \\
\bottomrule
\end{tabular}
\end{table}

\paragraph{Diagnostics.}
The rows of \Cref{tab:steady-summary} show one active cell-label grain
and zero grid-interface length in all ten final states. The mean first
quiescent-window start is $0.0036$ for $\alpha=0.3$ and $0.0019$ for
$\alpha=0.1$. Across seeds, these start times range over
$[0.0025,0.0045]$ and $[0.0015,0.0025]$, respectively.

The common time coordinate in \Cref{fig:grain-coarsening-curves} resolves
the rapid initial transition alongside the original-order curves.
After the first step, the mean active-grain count has fallen from $60$ to
$42.0$ for $\alpha=0.3$ and to $10.8$ for $\alpha=0.1$.
All five cell partitions have one active grain by step $8$ ($t=0.004$)
and step $4$ ($t=0.002$), respectively. From the following step onward,
every recorded changed-cell fraction is zero through $T=0.15$.
During the $\alpha=0.3$ transient, the recorded changed-cell fraction
varies nonmonotonically.

At fixed $h$, decreasing $\alpha$ changes both the L1 history and the
Helmholtz smoothing length. For the two additional orders, the lengths are
approximately $0.305$ and $0.671$ on the unit torus; see
\Cref{tab:steady-summary2}. The first update acts on $\chi^0$ with the corresponding smoothing length;
later updates use the evolving L1 history. These smoothing lengths are
comparable to the unit domain scale, placing the additional runs in a
large-smoothing regime.

The strong finite-level initial-state contribution is quantified by
\Cref{rem:small-alpha}: at step $300$, it still has weight about
$12.7\%$ for $\alpha=0.3$ and $50.9\%$ for $\alpha=0.1$.
In these runs, one-grain cell partitions form while the initial state
retains substantial history weight, with smoothing and redistribution
applied to the history average at every step. To separate the one-step smoothing scale,
one may hold $\ell_0$ fixed using
$h_\alpha=(\ell_0^2/\Gammafun(2-\alpha))^{1/\alpha}$ and compare update
counts, as in \Cref{fig:struct-memory-tail}.

The equal-time snapshots in \Cref{fig:grain-snapshots-equal-time} provide the corresponding geometric picture for the comparison seed.  At the same physical time, smaller $\alpha$ has already coarsened further in this fixed-step scaling.  {The target-selected snapshots in
\Cref{fig:grain-snapshots-matched-active} show the stored states nearest
the prescribed counts $50,40,30$ for the same seed. The actual counts vary across orders: the displayed
$\alpha=0.7$ counts are $60,25,25$, and the $\alpha=0.5$ counts are
$53,53,8$. Repeated panels result from selecting the same stored state
for two targets. These actual counts identify the coarsening stage represented by each panel.}

\begin{figure}[htp!]
    \centering
\includegraphics[page=5,width=.79\textwidth]{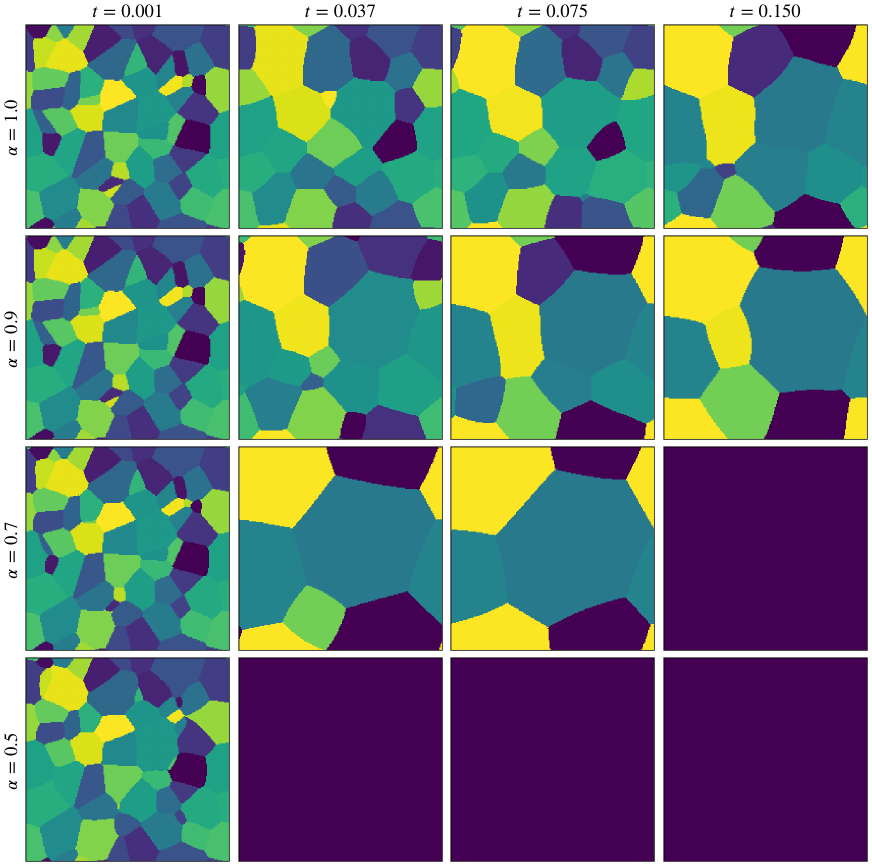}
\caption{Equal-time grain configurations for the same initial Voronoi seed on the periodic grid.  Rows correspond to $\alpha\in\{1.0,0.9,0.7,0.5\}$, and columns correspond to the displayed physical times.}
\label{fig:grain-snapshots-equal-time}
\end{figure}

\begin{figure}[htp!]
\centering
   \includegraphics[page=6,width=.59\textwidth]{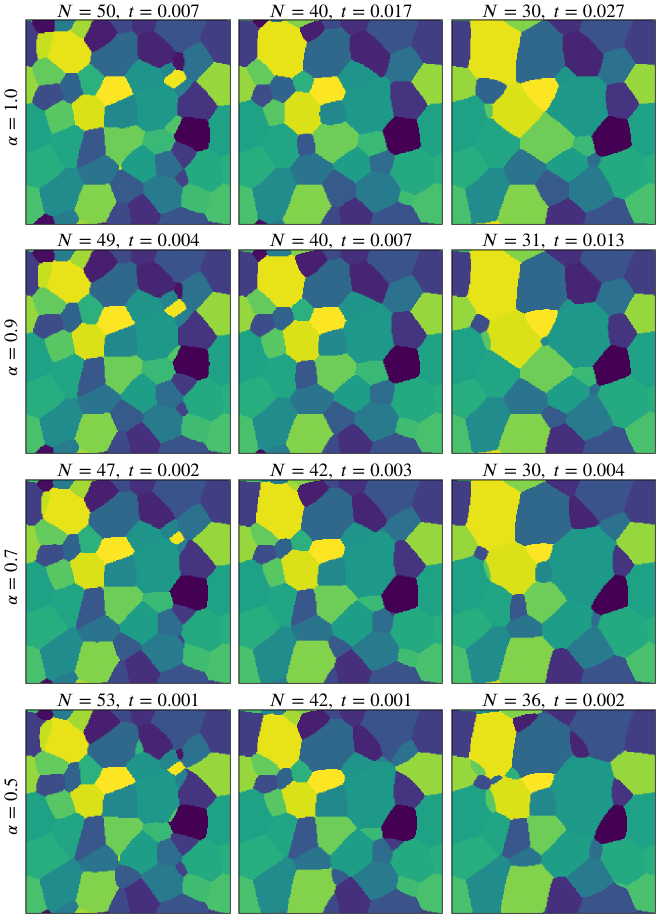}
\caption{{Grain configurations for seed $42$, selected from the stored
snapshots nearest the target counts $50,40,30$ (columns). Rows correspond
to $\alpha=1,0.9,0.7,0.5$. Actual counts and times are printed in each
panel; times are rounded to three decimal places. Nearest-state selection can assign the same stored state to two targets;
for example, the $\alpha=0.5$ row has actual counts $53,53,8$.}}
\label{fig:grain-snapshots-matched-active}
\end{figure}

\subsection{\texorpdfstring{{Grain-area and side-number statistics}}{Grain-area and side-number statistics}}

The statistical comparisons in
\Cref{fig:grain-normalized-area-hist,fig:grain-conditional-mean-area,fig:grain-vn-fractional}
use $\alpha\in\{1.0,0.9,0.7,0.3\}$ and the selection target
$N_{\rm gr}=30$. For each order and seed, the stored state nearest this
target is selected. The actual counts range from $29$ to $31$ for
$\alpha=1$, from $27$ to $35$ for $\alpha=0.9$, and from $25$ to $33$
for $\alpha=0.7$. These three orders are therefore approximately matched;
their statistics pool grains from five selected states at different times.
We include $\alpha=0.3$ to examine grain geometry at a smaller fractional
order with a larger smoothing length. Its selected states are all at step $1$, with counts
$40,37,39,46,48$ for seeds $42,52,62,72,82$, giving $210$ pooled grains.
These first-step samples have $37$--$48$ grains per seed and represent an
earlier coarsening stage than the three approximately matched orders.

The normalized-area histograms in \Cref{fig:grain-normalized-area-hist}
describe the empirical distribution of $A_i/\overline A$ in these samples.
The panels use common axes and retain the full exported bin range, including
the larger-area tail for $\alpha=0.3$. Differences involving that order
also reflect the different coarsening stage and first-step smoothing length.

\begin{figure}[htp!]
\centering
\includegraphics[page=7,width=.99\textwidth]{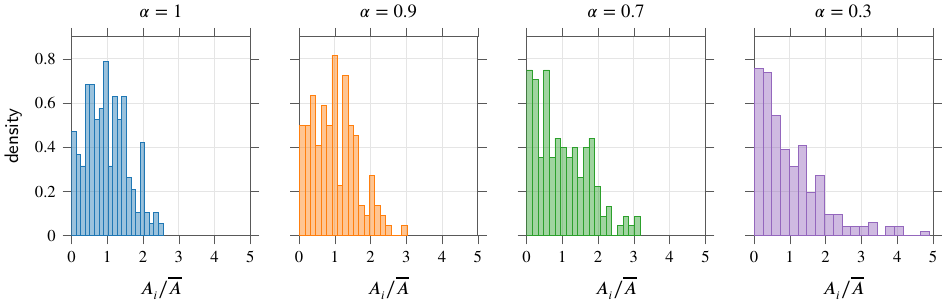}
\caption{{Empirical normalized-area distributions for
$\alpha=1,0.9,0.7,0.3$ in row order (left to right, then top to bottom),
pooling five states per order selected nearest $N_{\rm gr}=30$.
Actual count ranges are $29$--$31$, $27$--$35$, $25$--$33$, and $37$--$48$,
respectively. The $\alpha=0.3$ panel represents an earlier coarsening stage sampled at
the first step. Bin edges and densities are
those of the exported histograms; both axes are common to all four panels.}}
\label{fig:grain-normalized-area-hist}
\end{figure}

{\Cref{fig:grain-conditional-mean-area} shows side-number statistics
for the same retained samples. The conditional mean area generally increases with the number of sides,
with local decreases in some samples. For $\alpha=0.7$, it decreases from
approximately $2.094$ at nine sides (eight grains) to $1.635$ at ten sides
(one grain). The high-side classes contain few grains, as shown by the
accompanying frequency histogram.}

\begin{figure}[htp!]
\centering
\includegraphics[page=8,width=.99\textwidth]{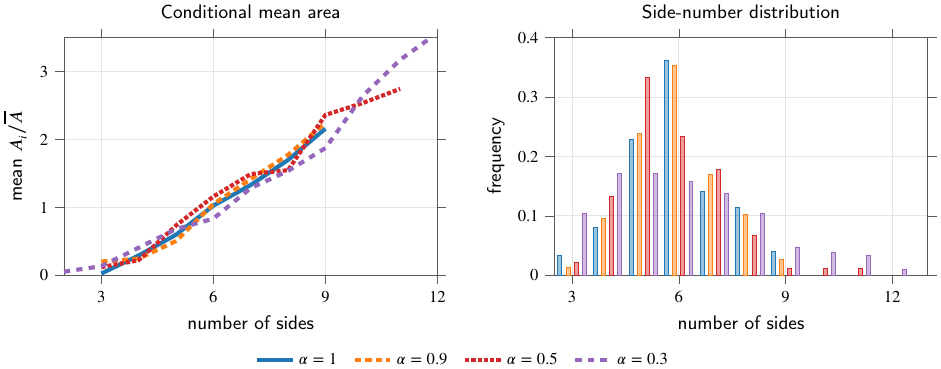}
\caption{{Side-number statistics for $\alpha=1,0.9,0.7,0.3$ using the same
samples as \Cref{fig:grain-normalized-area-hist}.
Left: conditional mean normalized grain area. Right: empirical
side-number frequency. The high-side classes contain few grains.}}
\label{fig:grain-conditional-mean-area}
\end{figure}

\subsection{Fractional von Neumann--Mullins diagnostic}

{This comparison uses $\alpha=1,0.9,0.7,0.3$ and the samples described
above. The first three orders are approximately matched in grain count;
the $\alpha=0.3$ panel displays the first-step sample. At this first step,
\begin{equation*}
 (\partial_h^\alpha A_i)^1
 =\frac{h^{1-\alpha}}{\Gammafun(2-\alpha)}
   \frac{A_i^1-A_i^0}{h}.
\end{equation*}
For $\alpha=0.3$ and the present $h$, the scalar prefactor is approximately
$0.0053813$. At step $1$, the fractional area derivative is therefore a scalar multiple
of the ordinary backward difference, so this panel relates first-step area
changes to side number.}

For each label $i$, let $A_i^m$ denote its discrete area at time step $m$.
The label is kept fixed throughout the simulation history; if the grain has
become extinct by time $m$, we set $A_i^m=0$.  Thus the discrete fractional
area derivative below is computed from the full history of the same label.  In
the scatter plots at time level $n$, we include only labels that are active at
that time, i.e. labels with $A_i^n>0$.

For each active label we compute the discrete fractional area derivative
\begin{equation}\label{eq:num-area-caputo}
(\partial_h^\alpha A_i)^n
=
\frac{1}{\Gammafun(2-\alpha)h^\alpha}
\sum_{k=0}^{n-1} b_k\bigl(A_i^{n-k}-A_i^{n-k-1}\bigr),
\qquad \alpha\in(0,1),
\end{equation}
and use the ordinary backward difference for $\alpha=1$.  We then plot this quantity against $s_i^n-6$, where $s_i^n$ is the discrete side number.  {Conditional on a suitable discrete chain rule for phase indicators
and a representation of the limiting interfacial memory as a Caputo-type
convolution of normal velocity, a formal fractional analogue of the
von Neumann--Mullins law would have the schematic form}
\begin{equation*}
(\partial_t^\alpha A_i)(t) \approx \kappa_\alpha\,\bigl(s_i(t)-6\bigr),
\end{equation*}
for a mobility-dependent constant $\kappa_\alpha$.  {The least-squares slopes in \Cref{fig:grain-vn-fractional} quantify the
empirical relation between the computed fractional area derivatives and the
centered side number $s_i^n-6$.} 

{\Cref{fig:grain-vn-fractional} shows the resulting four scatter plots,
with common axes and least-squares fits. For the $210$ active grains at
$\alpha=0.3$, the fitted slope is approximately $0.06374$.}  {Motivated by the classical von Neumann--Mullins law for two-dimensional
grain growth \cite{vonNeumann1952,Mullins1956}, this experiment examines the
relation between the computed fractional area derivatives and grain side numbers.}  The observed increasing trend shows that grains with more sides tend, on average, to have larger fractional area derivative than grains with fewer sides.

\begin{figure}[htp!]
\centering
\includegraphics[page=9,width=.99\textwidth]{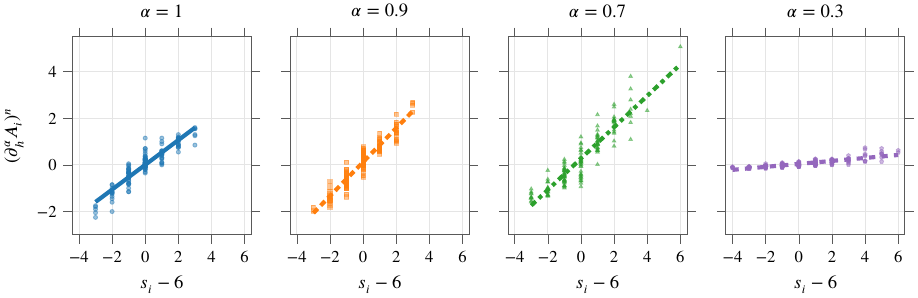}
\caption{{Fractional von Neumann--Mullins diagnostic for
$\alpha=1,0.9,0.7,0.3$ in row order, using the same samples as
\Cref{fig:grain-normalized-area-hist}. Points are individual active grains;
lines are least-squares fits. All four panels use identical axes, with the
vertical range containing all recorded derivatives.}}
\label{fig:grain-vn-fractional}
\end{figure}

\section{Conclusion and outlook}\label{sec:conclusion}

We introduced a history-aware fractional MBO scheme obtained by combining the
L1 discretization of the Caputo derivative with a Helmholtz-resolvent
thresholding mechanism.  {At each step, the Helmholtz resolvent acts on the convex L1 history
average of all previous thresholded states.}  The resulting
Helmholtz step is exactly one implicit L1 subdiffusion step, and the following
thresholding step projects the field back to pure phases.

{The scheme admits a variational characterization of its continuous
diffusion step, together with scalar maximum/comparison principles and
resolvent stability. Combined with the convex L1 history representation and
deterministic thresholding, these properties yield a well-defined
time-discrete construction. The memory-tail estimate quantifies the
influence of older states and yields a finite-dimensional pinning condition
under a contractive discrete resolvent. The simulations exhibit faster
coarsening for smaller fractional orders in the tested fixed-step regime,
where both history and smoothing vary. Comparisons at fixed smoothing
length examine the effect of the L1 history across successive updates on
the associated order-dependent time grids.}

{An important next step is to identify the limiting geometric law as
$h\downarrow0$. For interfaces with memory, defining a geometric term such as
$\mathcal I^{1-\alpha}V$ requires a transport representation of the
interfacial history. A convergence theory must combine this representation
with compactness, a no-loss condition for interfacial area, and identification
of the limiting memory term.}

Another natural direction is to
compare the present threshold-dynamics construction with diffuse-interface
models based on time-fractional Allen--Cahn equations \cite{DuYangZhou2020}.  A sharp-interface analysis connecting such Allen--Cahn dynamics with
memory-dependent interface motion would be complementary to the
thresholding approach considered here. Other possible extensions include volume-preserving fractional thresholding,
anisotropic kernels and surface tensions, adaptive choices of the time step, and
a variational convergence theory for the fractional minimizing-movement
structure suggested by the L1 history average.  

\setlength{\bibsep}{0.1\baselineskip}
\small
\bibliographystyle{abbrv}
\bibliography{references}

\end{document}